\documentclass[11pt]{article}
\usepackage{amsmath, amssymb, amscd, amsthm, amsfonts}
\usepackage{graphicx}
\usepackage{hyperref}
\usepackage{amssymb}
\usepackage{amsmath}
\usepackage{amsthm}
\usepackage{amsfonts}

\usepackage{a4wide}
\usepackage{longtable}
\usepackage{multirow}

\usepackage[latin1]{inputenc}
\usepackage[english]{babel}
\usepackage[T1]{fontenc}
\usepackage{amssymb,amsmath,amsthm,amstext,amsfonts,latexsym}

\usepackage{pgf,tikz}
\usepackage{mathrsfs}
\usetikzlibrary{arrows}

\newtheorem{theorem}{Theorem}[section]
\newtheorem{corollary}{Corollary}[section]
\newtheorem{conjecture}{Conjecture}[section]

\theoremstyle{definition}

\theoremstyle{remark}

\usepackage[all]{xy}
\title{$5$-rank of ambiguous class groups of\\ quintic Kummer extensions}
\author{Fouad ELMOUHIB \and Mohamed TALBI \and Abdelmalek AZIZI}
\date{}

\begin{document}

\maketitle

\begin{abstract}
Let $k \,=\, \mathbb{Q}(\sqrt[5]{n},\zeta_5)$, where $n$ is a positive integer $5^{th}$ power-free, whose $5-$class group denoted $C_{k,5}$ is isomorphic to $\mathbb{Z}/5\mathbb{Z}\times\mathbb{Z}/5\mathbb{Z}$. Let $k_0\,=\,\mathbb{Q}(\zeta_5)$ be the cyclotomic field containing  a primitive  $5^{th}$ root of unity $\zeta_5$. Let $C_{k,5}^{(\sigma)}$ the group of the ambiguous classes under the action of $Gal(k/k_0)$ = $\langle\sigma\rangle$. The aim of this paper is to determine  all integers $n$ such that the group of ambiguous classes $C_{k,5}^{(\sigma)}$ has rank $1$ or $2$.
\end{abstract}

\section{Introduction}\label{section-introduction}
One of the most important problems in number theory is the determination of the structure of class group of a number field, particularly its rank. The case of quadratic fields, Gauss's genus theory, determines the rank of $2$-class group. In a series of papers ( see [\ref{GER1}], [\ref{GER2}], [\ref{GER3}]), Frank Gerth III proved several results on the $3$-class groups of pure cubic extension of $\mathbb{Q}$ and cyclic cubic extension of $\mathbb{Q}$. Recently, S.Aouissi, M.C.Ismaili, M.Talbi, A.Azizi in [\ref{SIH}] had classified all fields $\mathbb{Q}(\sqrt[3]{n},\zeta_3)$ whose $3-$class group is of type $(9,3)$. \\
Let $k \,=\, \mathbb{Q}(\sqrt[5]{n},\zeta_5)$, a number of researchers have studied the $5$-class group $C_{k,5}$. M.Kulkarni, D.Majumdar and B.Sury in [\ref{Mani}] proved some results that are seen as generalisation of Gerth's work to the case of any odd prime, and they give more details in case of $5$-class group of $k$. In [\ref{Pa}], C.Parry gives a formula between class numbers of pure quintic field $\mathbb{Q}(\sqrt[5]{n})$ and its normal closure $k$. In \ref{Kobaya} H.Kobayashi proved that if the radicand $n$ has a prime factor $p$ congruent to $-1$ modulo $5$, then the class number of the pure quintic field $\Gamma\,=\,\mathbb{Q}(\sqrt[5]{n})$ is a multiple of $5$, and the class number of $k$ is multiple of $25$. \\
Let $n>1$ be a $5^{th}$ power-free integer and
$k\,=\,\mathbb{Q}(\sqrt[5]{n},\zeta_5)$ be a quintic Kummer extension of the cyclotomic field $k_0\,=\,\mathbb{Q}(\zeta_5)$. By $C_{k,5}^{(\sigma)}$ we denote the $5$-group of ambiguous ideal classes under the action of $Gal(k/k_0)$ = $\langle\sigma\rangle$, i.e $C_{k,5}^{(\sigma)}\,=\, \{\mathcal{A}\,\in\,C_{k,5}|\, \mathcal{A}^\sigma\,=\,\mathcal{A}\}$. Let $k^{\ast}\,=\,(k/k_0)^{\ast}$ be the maximal abelian extension of $k_0$ contained in the Hilbert $5$-class field $k_5^{(1)}$ of $k$,
which is called the relative $5$-genus field of $k/k_0$.

We consider the problem of finding the radicands \(n\) 
of all pure quintic fields $\Gamma\,=\,\mathbb{Q}(\sqrt[5]{n})$, for which the Galois group $\operatorname{Gal}(k^{\ast}/k)$ is non-trivial. The present work gives the complete solution of this problem
by characterizing all quintic Kummer extensions $k/k_0$ with  $5$-group of ambiguous ideal classes $C_{k,5}^{(\sigma)}$ of order $5$ or $25$. This paper can be viewed as the continuation of the work of M.Kulkarni, D.Majumdar and B.Sury in [\ref{Mani}]. \\
In fact, we shall prove the following Main Theorem:
\begin{theorem}
\label{thm:Rank1}
Let $\Gamma\,=\,\mathbb{Q}(\sqrt[5]{n})$ be a pure quintic field, where $n>1$ is a $5^{th}$ power-free integer, and $k\,=\,\mathbb{Q}(\sqrt[5]{n},\zeta_5)$ be its normal closure. We assume that the $5-$class group $C_{k,5}$ is of type $(5,5)$.\\
(1) If rank $(C_{k,5}^{(\sigma)})\,=\,1$, then the integer $n$ can be written in one of the following forms:
\begin{equation}
\label{eqn:Rank1}
 n\,=\,\left\lbrace
   \begin{array}{ll}
   
  5^{e}q_1^{2}q_2 \not \equiv \, \pm1\pm7\, (\mathrm{mod}\, 25) & \text{ with } \quad q_1 \,\text{ or }\, q_2$  $\not\equiv \,\pm7\, (\mathrm{mod}\, 25)\\
   
   5^ep \not\equiv \, \pm1,\pm7\, (\mathrm{mod}\,25) & \text{ with } \quad p\, \not\equiv \, -1\,(\mathrm{mod}\,25)\\
   
   5^{e}q_1 \not\equiv \, \pm1\pm7\, (\mathrm{mod}\, 25)& \text{ with }\quad q_1\,\equiv \,\pm7\, (\mathrm{mod}\, 25)\\
   
   p^{e}q_1 \, \equiv \, \pm1,\pm7\, (\mathrm{mod}\,25) & \text{ with } \quad p\, \not\equiv \, -1\,(\mathrm{mod}\,25),\, q_1\, \not\equiv \, \pm7\,(\mathrm{mod}\,25)\\

   p^{e} \, \equiv \, \pm1,\pm7\, (\mathrm{mod}\,25) &  \text{ with } \quad p\, \equiv \, -1\,(\mathrm{mod}\,25)\\
  
   q_1^{e_1}q_2\,\equiv \, \pm1\pm7\, (\mathrm{mod}\, 25) & \text{ with } \quad q_i\,\equiv \,\pm7\, (\mathrm{mod}\, 25)
   
   \end{array}
   \right.
\end{equation}
where $p\,\equiv \, -1\, (\mathrm{mod}\, 5)$ and $q_1,q_2\,\equiv \, \pm2\, (\mathrm{mod}\, 5)$ are primes and
$e, e_1$ are integers in $\{1,2,3,4\}$.\\\\
(2) If rank $(C_{k,5}^{(\sigma)})\,=\,2$, then the integer $n$ can be written in one of the following forms:
\begin{equation}
\label{eqn:Rank1}
 n\,=\,\left\lbrace
   \begin{array}{ll}
   5^el \not \equiv \, \pm1,\pm7\, (\mathrm{mod}\,25) & \quad \text{ with } \quad l\, \not\equiv \, 1\,(\mathrm{mod}\,25),\\

   l^{e_1}q_1 \, \equiv \, \pm1,\pm7\, (\mathrm{mod}\,25) & \quad \text{ with } \quad l\, \equiv \, 1\,(\mathrm{mod}\,5),\, q_1\, \equiv \, \pm2,\pm7,\pm3\,(\mathrm{mod}\,25)\\
   
   l^{e_1} \, \equiv \, \pm1,\pm7\, (\mathrm{mod}\,25) & \quad \text{ with } \quad l\, \equiv \, 1\,(\mathrm{mod}\,25),\\
   
   \end{array}
   \right.
\end{equation}
where $l\,\equiv \, 1\, (\mathrm{mod}\, 5)$ and  $q_1\,\equiv \, \pm2\, (\mathrm{mod}\, 5)$ are primes and
$e, e_1$ are integers in $\{1,2,3,4\}$.
\end{theorem}
This result will be underpinned by numerical examples obtained with the computational number theory system \textbf{PARI/GP} [\ref{PRI}] in section 3.


\begin{center}

$\textbf{Notations. \ }$
\end{center}
Throughout this paper, we use the following notations:
\begin{itemize}

\item The lower case letter $p$,$q$ and $l$ will denote a prime numbers such that, $p\,\equiv\,-1\, (\mathrm{mod}\, 5)$, $q\,\equiv\,\pm2\, (\mathrm{mod}\, 5)$ and $l\,\equiv\,1\, (\mathrm{mod}\, 5)$.

 \item $\Gamma\,=\,\mathbb{Q}(\sqrt[5]{n})$: a pure quintic field, where $n\neq 1$ is a $5^{th}$ power-free positive integer.
 
 \item $k_0\,=\,\mathbb{Q}(\zeta_5)$, the cyclotomic field, where $\zeta_5\,=\,e^{\frac{2i\pi}{5}}$ a primitive $5^{th}$ root of unity.
 
 \item $k\,=\,\mathbb{Q}(\sqrt[5]{n},\zeta_5)$: the normal closure of $\Gamma$, a quintic Kummer extension of $k_0$.
 
\item $\Gamma^{'},\,\Gamma^{''},\,\Gamma^{'''},\,\Gamma^{''''},\,$ the four conjugates quintic fields of $\Gamma$, contained in $k$. 
 
 \item $\langle\tau\rangle\,=\,\operatorname{Gal}(k/\Gamma)$ such that $\tau^4\,=\,id,\, \tau^3(\zeta_5)\,=\,\zeta_5^3,\, \tau^2(\zeta_5)\,=\,\zeta_5^4,\, \tau(\zeta_5)\,=\,\zeta_5^2$ and $\tau(\sqrt[5]{n})\,=\,\sqrt[5]{n}$.
 
 \item $\langle\sigma\rangle\,=\,\operatorname{Gal}(k/k_0)$ such that $\sigma^5\,=\,id,\ \sigma(\zeta_5)\,=\,\zeta_5$ and $\sigma(\sqrt[5]{n})\,=\,\zeta_5\sqrt[5]{n},\, \sigma^2(\sqrt[5]{n})\,=\,\zeta_5^2\sqrt[5]{n},\,\\\\ \sigma^3(\sqrt[5]{n})\,=\,\zeta_5^3\sqrt[5]{n},\, \sigma^4(\sqrt[5]{n})\,=\,\zeta_5^4\sqrt[5]{n} $.
 
 \item $\lambda\,=\,1-\zeta_5$ is prime element above $5$ of $k_0$.
 
 \item $q^{\ast}\,=\,0,\,1$ or $2$ according to whether $\zeta_{5}$ and $1+\zeta_5$ is not norm or is norm of an element of\\\\ $k^*\,=\,k\setminus\lbrace 0\rbrace$.
 
 \item $d$: the number of prime ideals of $k_{0}$ ramified in $k$.
 
 \item For a number field $L$, denote by:

  \begin{itemize}
   \item $\mathcal{O}_{L}$: the ring of integers of $L$;
   \item $E_{L}$: the group of units of $L$;
   \item $C_{L}$, $h_{L}$, $C_{L,5}$: the class group, class number, and $5$-class group of $L$.
   \item $L_5^{(1)}, L^{\ast}$: the Hilbert $5$-class field of $L$, and the absolute genus field of $L$. 
  \end{itemize}
  
\end{itemize}
\begin{center}
\begin{tikzpicture}
\begin{scope}[xscale=2,yscale=2]
  
  \node (A) at (0,3) {$\mathbf{k}$};
  \node (B) at (1,2){$\mathbf{\Gamma}$ };
  \node (C) at (1.5,2){$\mathbf{\Gamma'}$ };
  \node (D) at (2,2){$\mathbf{\Gamma''}$ };
  \node (E) at (2.5,2){$\mathbf{\Gamma'''}$ };
  \node (F) at (3,2){$\mathbf{\Gamma''''}$ };  
  \node (G) at (-1,1) {$\mathbf{k_0}$};
  \node (H)   at (0,0) {$\mathbb{Q}$ };
  \node (I)   at (0,-0.5) {\underline{Figure 1}};
  \draw [-,>=latex] (G) -- (A)   node[midway,above,rotate=40] {5};
  \draw [-,>=latex] (B) -- (A)   node[midway,above,rotate=-40] {4};
  \draw [-,>=latex] (C) -- (A)   node[midway,below left] {};
  \draw [-,>=latex] (D) -- (A)   node[midway,below left] {};
  \draw [-,>=latex] (E) -- (A)   node[midway,below left] {};
  \draw [-,>=latex] (F) -- (A)   node[midway,below left] {};
  \draw [-,>=latex] (H) -- (G)   node[midway,below right] {4};
  \draw [-,>=latex] (H) -- (B)   node[midway,below right] {5};
  \draw [-,>=latex] (H) -- (C)   node[midway,below right] {};
  \draw [-,>=latex] (H) -- (D)   node[midway,below right] {};
  \draw [-,>=latex] (H) -- (E)   node[midway,below right] {};
  \draw [-,>=latex] (H) -- (F)   node[midway,below right] {};

\end{scope}

\end{tikzpicture}

\end{center}




\section{\Large Proof of Main Theorem}
\begin{theorem}{(Decompositon in cyclotomic fields)}\\
Let $m$ a positive integer and $p$ a prime number. Suppose $p$ do not divides $m$, and let $f$ be the smallest positive integer such that $p^f \, \equiv \, 1\, (\mathrm{mod}\, m)$. Then $p$ splits into $\phi(m)/f$ distinct primes in $\mathbb{Q}(\zeta_m)$ each of which has a residue class degree $f$. In particular, p splits completely if and only if $p \, \equiv \, 1\, (\mathrm{mod}\, m)$
\end{theorem}
\begin{proof}
see [\ref{cyc}] page 14.
\end{proof}
\begin{corollary}
\label{deco}
Let $p$ a prime integer, we have :
\item[-] If $p\,=\,5$, then $\lambda\,=\,1-\zeta_5$ is the unique prime over 5 in $\mathbb{Q}(\zeta_5)$.
\item[-] If $l \, \equiv \, 1\, (\mathrm{mod}\, 5)$, then $l$ splits completely in $\mathbb{Q}(\zeta_5)$ as $l\,=\,\pi_1\pi_2\pi_3\pi_4$, with $\pi_i$ are primes in $\mathbb{Q}(\zeta_5)$
\item[-] If $q \, \equiv \, \pm2\,(\mathrm{mod}\, 5)$, then $q$ is inert in $\mathbb{Q}(\zeta_5)$.
\item[-] If $p \, \equiv \, -1\, (\mathrm{mod}\, 5)$, then $p$ splits in $\mathbb{Q}(\zeta_5)$ as $p\,=\,\pi_1\pi_2$, with $\pi_i$ are primes in $\mathbb{Q}(\zeta_5)$.
\end{corollary}
Before proving the main theorem, we give a proof of the existance of a unique prime $l \, \equiv \, 1\,(\mathrm{mod}\,5)$ divides the radicand $n$ in the case of rank $(C_{k,5}^{(\sigma)})\,=\,2$
\begin{theorem}
\label{ran2} 
If rank $C_{k,5}^{(\sigma)} = 2$, so  $C_{k,5}\,=\,C_{k,5}^{(\sigma)}$, and there is unique prime $l \, \equiv \, 1\,(\mathrm{mod}\,5)$ that divides the radicand  $n$. Furthermore, we have $(k/k_0)^*\,=\, k_5^{(1)}$.
\end{theorem}
\begin{proof}
If rank $(C_{k,5}^{(\sigma)}) = 2$ so the order of $C_{k,5}^{(\sigma)}$ is at least $25$, since $C_{k,5}^{(\sigma)} \subseteq C_{k,5}$ and $|C_{k,5}|\,=\,25$ because $C_{k,5}$ is of type $(5,5)$ we have $C_{k,5}\,=\,C_{k,5}^{(\sigma)}$ it means that all ideal classes are ambiguous.\\
By class fields theory $C_{k,5}^{1-\sigma}$ correspond to $(k/k_0)^*$ and $Gal(k_5^{(1)}/k)\cong C_{k,5}$. Since $C_{k,5}^{(\sigma)}\,=\,C_{k,5}$, we get $C_{k,5}^{1-\sigma}\,=\,\{1\}$, hence $(k/k_0)^*\,=\, k_5^{(1)}$, and by [\ref{Mani}, proposition 5.8] we know explicitly in this case the Hilbert $5-$class field of $k$.\\
We assume now that there is no prime $l \, \equiv \, 1\,(\mathrm{mod}\,5)$ divides $n$.\\ We can write $n$ as $n\,=\,5^eq_1^{f_1}...q_r^{f_r}p_1^{g_1}....p_s^{g_s}$ with $q_i\, \equiv \, \pm2 \, (mod\, 5)$ and $p_j\, \equiv \, -1\, (mod\, 5)$,\\ $f_i\,=\,1,2,3,4$ for $1\leq i\leq r$ and $g_j\,=\,1,2,3,4$ for $1\leq j\leq s$, and $e\,=\,0,1,2,3,4$, by Corollary \ref{deco} each $q_i$ is inert in $k_0$, and $q_i$ is ramified in $\Gamma\,=\,\mathbb{Q}(\sqrt[5]{n})$, also by Corollary \ref{deco} $p_j$ splits in $k_0$ as $p_j\,=\,\pi_1\pi_2$, where $\pi_1,\pi_2$ are primes in $k_0$, and $p_j$ is ramified in $\Gamma\,=\,\mathbb{Q}(\sqrt[5]{n})$, so the prime ideals ramified in $k/k_0$ are those above $q_i$ and $\pi_j$ and the ideal above $\lambda$ with $\lambda\,=\,1-\zeta_5$ if 5 is ramified in $\Gamma\,=\,\mathbb{Q}(\sqrt[5]{n})$.\\
If $\lambda$ is ramified in $k/k_0$, we note $\mathcal{I}$ the prime ideal in $k$ above $\lambda$, and for $1\leq i\leq r$ ,$\mathcal{Q}_i$ the prime ideal above $q_i$ in $k$, and for $1\leq j\leq s$ ,$\mathcal{P}_j$ the prime ideal above $\pi_j$ in $k$, with $\pi_j$ is prime of $k_0$ above $p_j$. We have evidently $\mathcal{I}^5\,=\,(\lambda)$, $\mathcal{Q}_i^5\,=\,(q_i)$, $\mathcal{P}_j^5\,=\,(\pi_j)$ in $k$.\\ 
we note by $C_{k,s}^{(\sigma)}$ the group of strong ambiguous adeal classes. We have to treat two cases:\\

(i) $C_{k,s}^{(\sigma)}\,\neq\, C_{k,5}^{(\sigma)}\,=\,C_{k,5}$:\\

Let $C_{k,5}^+\,=\,\{\mathcal{A}\in\, C_{k,5}| \mathcal{A}^{\tau^2}\,=\,\mathcal{A}\}$ and $C_{k,5}^-\,=\,\{\mathcal{A}\in\, C_{k,5}| \mathcal{A}^{\tau^2}\,=\,\mathcal{A}^{-1}\}$ be a nontrivial subgoups of $C_{k,5}$. We have $(C_{k,5}^{(\sigma)})^+\,=\, C_{k,5}^+$, by [\ref{Mani}, Lemma 6.2] $C_{k,5}^+\, \simeq\, C_{\Gamma,5}$ i.e $C_{k,5}^+$ can be generated by $5-$class comes from $\Gamma$ ($|C_{k,5}^+|\,=\,5$). The strong ambiguous classes are those of primes ramified in $k/k_0$, namly $[\mathcal{Q}_i]$ for $1\leq i\leq r$, $[\mathcal{P}_j]$ for $1\leq j\leq s$ and $[\mathcal{I}]$ if $\lambda$ is ramified in $k/k_0$. Its easy to see that:  $[\mathcal{Q}_i^{\tau^2}]\,=\,[\mathcal{Q}_i]^{\tau^2}\,=\,[\mathcal{Q}_i]$ and $[\mathcal{P}_j^{\tau^2}]\,=\,[\mathcal{P}_j]^{\tau^2}\,=\,[\mathcal{P}_j]$ and $[\mathcal{I}^{\tau^2}]\,=\,[\mathcal{I}]^{\tau^2}\,=\,[\mathcal{I}]$, we now that $C_{k,5}/C_{k,s}^{(\sigma)}$ is generated by image in $C_{k,5}/C_{k,s}^{(\sigma)}$  of element in $C_{k,5}^+$. Since $C_{k,s}^{(\sigma)}$ is generated by $[\mathcal{Q}_i],[\mathcal{P}_j]$ and $[\mathcal{I}]$ if $\lambda$ is ramified in $k/k_0$, all elements of $C_{k,5}$ will be fixed by $\tau^2$, in particular whose of $C_{k,5}^-$, therefore $\forall \mathcal{A} \in C_{k,5}^-$, $\mathcal{A}^{\tau^2}\,=\,\mathcal{A}^{-1}\,=\,\mathcal{A}$ i.e $\mathcal{A}^2\,=\,1$ i.e $\mathcal{A}^4\,=\,1$, hence $\mathcal{A}\,=\,1$ because $\mathcal{A}$ is $5-$class, so we get $C_{k,5}^-\,=\,{1}$, and this contradict the fact that $|C_{k,5}^{-}|=5$.\\

(ii) $C_{k,s}^{(\sigma)}\,=\, C_{k,5}^{(\sigma)}\,=\,C_{k,5}$:\\

In this case $C_{k,5}$ will be generated by $[\mathcal{Q}_i]$, $[\mathcal{P}_j]$ and $[\mathcal{I}]$ if $\lambda$ is ramifed in $k/k_0$, and as in (i) all the classes are fixed by $\tau^2$, which give the same contradiction. Thus we proved the existence of a prime $l$ divides $n$ such that $l\,\equiv\, 1\, (\mathrm{mod}\,5)$.\\ 
According to [\ref{Mani},section 5.1], we have rank $C_{k,5}^{(\sigma)}\,=\, d-3+q^*$ where $d$ is the number of prime ramified in $k/k_0$ and $q^*$ is an index has as value 0,1 or 2. Assuming that there is two prime $l_1$ and $l_2$ divides $n$ such that $l_i\,\equiv\, 1\, (\mathrm{mod}\,5)$, $(i = 1,2)$, then $d \geq 8$ and rank $C_{k,5}^{(\sigma)}$ is at least $5$, that is impossible. Thus if rank $C_{k,5}^{(\sigma)}\,=\,2$ its exsite unique prime $l\,\equiv\, 1\, (\mathrm{mod}\,5)$ divides $n$. 
\end{proof}
\subsection{Proof of Theoreme \ref{thm:Rank1}}

Let $\Gamma\,=\,\mathbb{Q}(\sqrt[5]{n})$ be a pure quintic field, where $n\geq 2$ is a $5^{th}$ power-free integer, $k\,=\,\Gamma(\zeta_5)$ be its normal closure, and $C_{k,5}$ be the $5$-class group of $k$. Let $C_{k,5}^{(\sigma)}$ be the ambigous ideal classes group under the action of $Gal(k/k_0)\,=\,\langle\sigma\rangle$. Since $k_0$ has class number $1$, $C_{k,5}^{(\sigma)}$ is un elementary abelian $5$-group, so rank $(C_{k,5}^{(\sigma)})\,=\,1 \, \mathrm{or}\,2$.\\
According to [\ref{Mani},section 5.1], the rank of $C_{k,5}^{(\sigma)}$ is given as follows:
\begin{center}
rank $(C_{k,5}^{(\sigma)})\,=\, d-3+q^*$
\end{center}
where $d$ is the number of prime ideals of $k_{0}$ ramified in $k$, and $q^{\ast}\,=\,0,\,1$ or $2$ according to whether $\zeta_{5}$ and $1+\zeta_5$ is not norm or is norm of an element of $k^*\,=\,k\setminus\lbrace 0\rbrace$ as follows:
\begin{center}
$q^*$ = $ \begin{cases}
2 & \text{if }\, \zeta, 1+\zeta \in N_{k/k_0}(k^*),\\
1 & \text{if }\, \zeta^i(1+\zeta)^j \in N_{k/k_0}(k^*)\,\text{ for some i and j },\\
0 & \text{if }\, \zeta^i(1+\zeta)^j \notin N_{k/k_0}(k^*)\, \text{for}\hspace{2mm} 0\leq i,j\leq 4 \text{ and}\hspace{2mm} i+j\neq 0.\\
\end{cases}$
\end{center}
We can writ $n$ as $n\,=\,\mu\lambda^{e}\pi_1^{e_1}....\pi_g^{e_g}$, where $\mu$ is a unit in $\mathcal{O}_{k_0}$, $\lambda = 1-\zeta_5$, $\pi_1,,,,\pi_g$ are primes in $k_0$ and $e \in \{0,1,2,3,4\}$, $e_i \in \{1,2,3,4\}$ for $1\leq i\leq g$. According to [\ref{Mani}, Lemma 5.1] we have, $\zeta_5\, \in N_{k/k_0}(k^*)\, \Longleftrightarrow\, N_{k_0/\mathbb{Q}}((\pi_i))\, \, \equiv \, 1 \,(\mathrm{mod}\,25)$ for all i, and from [\ref{Jan}, proposition 8.2] if $\pi$ is a prime of $\mathcal{O}_{ k_0}$ over a prime $p\,\in\, \mathbb{Z}$ we have $N_{k_0/\mathbb{Q}}((\pi))\, =\, p^f$, with $f$ is the least positif integer such that $p^f\,\, \equiv \,\,1\, (\mathrm{mod},\,5)$, so we can get that $\zeta_5$ is norm of an element of $k^*\,=\,k\setminus\lbrace 0\rbrace$ if and only if $p^f\,\, \equiv \,\,1\, (\mathrm{mod},\,25)$ for all prime $p\,\neq\,5$ dividing $n$:
\begin{enumerate}
\item[$-$] If $l\,\, \equiv \,\,1\, (\mathrm{mod},\,5)$, by corollary \ref{deco} $l\,=\,\pi_1\pi_2\pi_3\pi_4$, we have $N_{k_0/\mathbb{Q}}((\pi_i))\, =\, l$, so to have $N_{k_0/\mathbb{Q}}((\pi_i))\, \, \equiv \,\, 1\, (\mathrm{mod},\,25)$ the prime $l$ must verify $l\,\, \equiv \,\,1\, (\mathrm{mod}\,25)$.

\item[$-$] If $q\,\, \equiv \,\,\pm2\, (\mathrm{mod},\,5)$ we have $q$ is inert in $k_0$, so $N_{k_0/\mathbb{Q}}((q))\, =\, q^4$, so to have $N_{k_0/\mathbb{Q}}((q))\, \, \equiv \,\, 1\, (\mathrm{mod},\,25)$ the prime $q$ must verify $q\, \equiv \,\pm7\, (\mathrm{mod}\,25)$.

\item[$-$] If $p\,\, \equiv \,\,-1\, (\mathrm{mod}\,5)$,by corollary \ref{deco} $p\,=\,\pi_1\pi_2$ we have $N_{k_0/\mathbb{Q}}((\pi))\, =\, p^2$, so to have $N_{k_0/\mathbb{Q}}((\pi))\, \, \equiv \,\, 1\, (\mathrm{mod}\,25)$ the prime $p$ must verify $p\,\, \equiv \,\,-1\, (\mathrm{mod}\,25)$.
\end{enumerate}
(1) If rank $(C_{k,5}^{(\sigma)})\,=\,1$, we get that $d+q^*\,=\,4$, so there are three possible cases as follows:
\begin{itemize}
    \item Case 1: \(q^*=0\,\, \mathrm{and}\,\, d=4\),
    \item Case 2: \(q^*=1\,\, \mathrm{and}\,\, d=3\),
    \item Case 3: \(q^*=2\,\, \mathrm{and}\,\, d=2\), \end{itemize}
We will successively treat the three cases to prove the first point of the main theorem.

\begin{itemize}
 \item Case 1: we have $q^* = 0$ and $d=4$, so the number of prime ideals which are ramified in $k/k_0$ should be $4$. According to the proof of theorem \ref{ran2}, if $n$ is divisible by a prime $l \, \equiv \, 1\, (\mathrm{mod}\, 5)$ then the prime $l$ is unique.\\\\
 - If $l \, \equiv \, 1\, (\mathrm{mod}\, 5)$ divides $n$, then by Corollary \ref{deco}, $l\,=\,\pi_1\pi_2\pi_3\pi_4$ where $\pi_i$ are primes of $k_0$. The prime $l$ is ramified in $\Gamma$, because disk$(\Gamma/\mathbb{Q})\,=\, 5^5n^4$ and $l$ divides this discriminent, then $\pi_1,\pi_2,\pi_3$ and $\pi_4$ are ramified in $k$. Hence we have $d=4$, so $l$ is the unique prime divides $n$, because if $n$ is dividing by other prime we obtain $d>4$, which is impossible in the three cases. So $n\,=\,l^{e_1}$ with  $l \, \equiv \, 1\, (\mathrm{mod}\, 5)$ and $e_1\in\{1,2,3,4\}$. According to [\ref{Mani}, Lemma 5.1] we have $(\lambda)$ ramifies in $k/k_0\, \Longleftrightarrow \, n\not \equiv \, \pm1\pm7\, (\mathrm{mod}\,25)$, with $\lambda\,=\,1-\zeta_5$, so in the case $n\,=\,l^{e_1}$
where  $l \, \equiv \, 1\, (\mathrm{mod}\, 5)$ we should have $n \, \equiv \, \pm1\pm7\, (\mathrm{mod}\,25)$, so the only $l$ verifiy $n\,=\, l^{e_1} \, \equiv \, \pm1\pm7\, (\mathrm{mod}\,25)$ is $l \, \equiv \, 1\, (\mathrm{mod}\, 25)$, and for this prime we have $q^*\,\geq\,1$, because $\zeta_5$ is norme, which is impossible in this case.\\
We note that if $n\,=\,l^{e_1}$ with  $l \, \equiv \, 1\, (\mathrm{mod}\, 25)$ and $q*\,=\,2$, there is no fields $k$ of type $(5,5)$, because we have rank$(C_{k,5}^{(\sigma)})\,=\,3$, and if $q^*\,=\,1$ we have rank$(C_{k,5}^{(\sigma)})\,=\,2$, which will treat in the second point of the proof.\\\\
-If no prime $l \, \equiv \, 1\, (\mathrm{mod}\, 5)$ divides $n$, we have two forms of $n$:

\begin{enumerate}
\item[$(i)$] $n\,=\, 5^{e}p^{e_1}q_1^{e_2} \not \equiv \, \pm1\pm7\, (\mathrm{mod}\, 25)$ with $p \, \equiv \, -1\, (\mathrm{mod}\, 5),\,\, q_1 \, \equiv \, \pm2\, (\mathrm{mod}\, 5)$ and $e,e_1,e_2 \in \{1,2,3,4\}$. By Corollary \ref{deco}, $p\,=\,\pi_1\pi_2$, where $\pi_1,\,\pi_2$ are primes in $k_0$ and $q_1$ is inert in $k_0$, the prime $p$  is ramified in $\Gamma$, then $\pi_1,\,\pi_2$ are ramified in $k$. We have $q_1$ is ramified in $\Gamma$. Since $n\,\not \equiv \, \pm1\pm7\, (\mathrm{mod}\, 25)$, then $\lambda\,=\,1-\zeta_5$ is ramified in $k$, so we get $d=4$. To verifiy that $n\,=\, 5^{e}p^{e_1}q_1^{e_2} \not \equiv \, \pm1\pm7\, (\mathrm{mod}\, 25)$ we can choose $e_1\,=\,2$ and $e_2\,=\,1$ because $\mathbb{Q}(\sqrt[5]{ab^2c^3d^4})\,=\,\mathbb{Q}(\sqrt[5]{a^2b^4cd^3})\,=\,\mathbb{Q}(\sqrt[5]{a^3bc^4d^2})\,=\,\mathbb{Q}(\sqrt[5]{a^4b^3c^2d})$, so $n\,=\, 5^{e}p^{2}q_1$ with $e \in \{1,2,3,4\}$. In one hand if $e\,=\,2,3,4$ we have $n\, \equiv \, 0\, (\mathrm{mod}\, 25)$, in the other hand if $n\,=\, 5p^{2}q_1$ we have $p^2q_1\,=\,5\alpha+2$ or $p^2q_1\,=\,5\alpha'+3$ with $\alpha,\alpha' \in \mathbb{Z}$, so $5p^{2}q_1 \, \equiv \, 10\, (\mathrm{mod}\, 25)$ or $5p^{2}q_1\, \equiv \, 15\, (\mathrm{mod}\, 25)$, therefore we conclude that $n\, \not \equiv \, \pm1\pm7\, (\mathrm{mod}\, 25)$. According to [\ref{Mani}, Theorem 5.18], if $p\,\equiv\,-1\,(\mathrm{mod}\, 25)$ and $q_1\,\equiv\,\pm7\,(\mathrm{mod}\, 25)$, rank $(C_{k,5})$ is at least $3$ which contradict the fact that $C_{k,5}$ is of type $(5,5)$, and by the proof of [\ref{Mani}, Theorem 5.13, Theorem 5.15], for the other congruence cases of $p$ and $q_1$ we have $q^*\,=\,1$ which is impossible in this case.\\

\item[$(ii)$] $n\,=\, p^{e}q_1^{e_1}q_2^{e_2} \, \equiv \, \pm1\pm7\, (\mathrm{mod}\, 25)$ with $p \, \equiv \, -1\, (\mathrm{mod}\, 5),\,\, q_1 \, \equiv \, 2\, (\mathrm{mod}\, 5),\,\, \\q_2 \, \equiv \, 3\, (\mathrm{mod}\, 5)$ and $e,e_1,e_2 \in \{1,2,3,4\}$. As $\mathbb{Q}(\sqrt[5]{ab^2c^3d^4})\,=\,\mathbb{Q}(\sqrt[5]{a^2b^4cd^3})\,=\,\mathbb{Q}(\sqrt[5]{a^3bc^4d^2})\,=\,\mathbb{Q}(\sqrt[5]{a^4b^3c^2d})$ we can choose  $e_1\,=\,2, e_2\,=\,1$ i.e $n\,=\, p^{e}q_1^{2}q_2 $ with $e \in \{1,2,3,4\}$. By Corollary \ref{deco} $p\,=\,\pi_1\pi_2$ where $\pi_1,\,\pi_2$ are primes in $k_0$ and $q_1,\, q_2$ are inert in $k_0$. We have $p,\,q_1$ and $q_2$ are ramified in $\Gamma$,so $\pi_1,\,\pi_2,\,q_1$ and $q_2$ are ramified in $k$ then $d=4$. The condition $n\,=\, p^{e}q_1^{2}q_2 \, \equiv \, \pm1\pm7\, (\mathrm{mod}\, 25)$ is not verified for all $p \, \equiv \, -1\, (\mathrm{mod}\, 5),\,\, q_1 \, \equiv \, 2\, (\mathrm{mod}\, 5)$ and $q_2 \, \equiv \, 3\, (\mathrm{mod}\, 5)$, so we combine all the cases of congruence and we obtain that $p \, \equiv \, -1\, (\mathrm{mod}\, 25),\,\, q_1 \, \equiv \, 12\, (\mathrm{mod}\, 25),\,\, q_2 \, \equiv \, 3\, (\mathrm{mod}\, 25)$. By [\ref{Mani}, Lemma 5.1], since $N_{k_0/\mathbb{Q}}(\pi_i)\,\equiv\,1\,(\mathrm{mod}\, 25),\, N_{k_0/\mathbb{Q}}(q_1)\,\not\equiv\,1\,(\mathrm{mod}\, 25), N_{k_0/\mathbb{Q}}(q_2)\,\not\equiv\,1\,(\mathrm{mod}\, 25)$, we have $q^*\,=\,1$ which is impossible in this case.

\end{enumerate}

We deduce that in the case 1, there is no radicand $n$ who verifiy rank$(C_{k,5}^{(\sigma)})\,=\,1$.

\item Case 2: we have $q^* = 1$ and $d=3$, so the number of prime ideals which are ramified in $k/k_0$ should be $3$. According to case 1, n is not divisible by any prime $l \, \equiv \, 1\, (\mathrm{mod}\, 5)$ in this case.\\
We can writ $n$ as $n\,=\,\mu\lambda^{e}\pi_1^{e_1}....\pi_g^{e_g}$, where $\mu$ is a unit in $\mathcal{O}_{k_0}$, $\lambda = 1-\zeta_5$, $\pi_1,,,,\pi_g$ are primes in $k_0$ and $e \in \{0,1,2,3,4\}$, $e_i \in \{1,2,3,4\}$ for $1\leq i\leq g$.\\
By [\ref{Mani}, proposition 5.2] $d\,=\,g$ or $g+1$ according to whether $n \, \equiv \, \pm1\pm7\, (\mathrm{mod}\, 25)$ or $n \not \equiv \, \pm1\pm7\, (\mathrm{mod}\, 25)$, to obtain $d=3$, $n$ must be written in $\mathcal{O}_{k_0}$ as: $n\,=\, \pi_1^{e_1}\pi_2^{e_2}\pi_3^{e_3}$ or $n\,=\, \lambda^{e}\pi_1^{e_1}\pi_2^{e_2}$, therefore we have three forms of $n$: 
\begin{enumerate}
\item[$(i)$] $n\,=\, 5^ep^{e_1} \not \equiv \,\pm1\pm7\, (\mathrm{mod}\, 25)$ with $p \, \equiv \, -1\, (\mathrm{mod}\, 5)$ and $e,e_1 \in \{1,2,3,4\}$. As $\mathbb{Q}(\sqrt[5]{ab^2c^3d^4})\,=\,\mathbb{Q}(\sqrt[5]{a^2b^4cd^3})\,=\,\mathbb{Q}(\sqrt[5]{a^3bc^4d^2})\,=\,\mathbb{Q}(\sqrt[5]{a^4b^3c^2d})$ we can choose  $e_1\,=\,1$ i.e $n\,=\, 5^ep $ with $e \in \{1,2,3,4\}$. By Corollary \ref{deco} $p\,=\, \pi_1\pi_2$ with $\pi_1,\,\pi_2$ are primes in $k_0$, we have $p$ is ramified in $\Gamma$, so $\pi_1,\,\pi_2$ are ramified in $k$, and since $n\,\not \equiv \,\pm1\pm7\, (\mathrm{mod}\, 25)$, according to [\ref{Mani}, Lemma 5.1], $\lambda \,=\, 1-\zeta_5$ is ramified in $k$ so we obtain $d=3$. The condition $n\,\not \equiv \,\pm1\pm7\, (\mathrm{mod}\, 25)$ is verified for all $p \, \equiv \, -1\, (\mathrm{mod}\, 5)$ because, if $e\,=\,2,3,4$ we have $n\,=\,5^ep\, \equiv \, 0\, (\mathrm{mod}\, 25)$, if $e=1$ i.e $n\,=\, 5p$ we have $p\,=\,5\alpha+4$ that implie $5p\,=\,25\alpha+20$ with $\alpha \in \mathbb{Z}$, so $n\,=\,5p \, \equiv \, 20\, (\mathrm{mod}\, 25)$. According to the proof of [\ref{Mani}, theorem 5.15] if $p \, \equiv \, -1 (\mathrm{mod}\, 25)$ we have $q^*\,=\,2$, and   if $p \, \not\equiv \, -1 (\mathrm{mod}\, 25)$ we have $q^*\,=\,1$, we conclude that $n\,=\,5^ep \not \equiv \, \pm1\pm7\, (\mathrm{mod}\, 25)$ with $p \, \not\equiv \, -1 (\mathrm{mod}\, 25)$.\\
We note that the computational number theory system \textbf{PARI} [\ref{PRI}], show that if $n\,=\,5^ep \not \equiv \, \pm1\pm7\, (\mathrm{mod}\, 25)$ with $p \, \not\equiv \, -1 (\mathrm{mod}\, 25)$, the field $k$ is not always of type $(5,5)$.

\item[$(ii)$] $n\,=\, 5^{e}q_1^{e_1}q_2^{e_2} \not \equiv \, \pm1\pm7\, (\mathrm{mod}\, 25)$ with $q_1 \, \equiv \, 2\, (\mathrm{mod}\, 5),\,\, q_2 \, \equiv \, 3\, (\mathrm{mod}\, 5)$ and $e,e_1,e_2 \in \{1,2,3,4\}$. As $\mathbb{Q}(\sqrt[5]{ab^2c^3d^4})\,=\,\mathbb{Q}(\sqrt[5]{a^2b^4cd^3})\,=\,\mathbb{Q}(\sqrt[5]{a^3bc^4d^2})\,=\,\mathbb{Q}(\sqrt[5]{a^4b^3c^2d})$, we can choose $e_1\,=\,2$ and $e_2\,=\,1$ i.e $n\,=\, 5^{e}q_1^{2}q_2$ with $e \in \{1,2,3,4\}$. By Corollary \ref{deco}, $q_1$ and $q_2$ are inert in $k_0$ , and $q_1$, $q_2$  are ramified in $\Gamma$, then $q_1,\,q_2$ are ramified in $k$. Since $n\,\not \equiv \, \pm1\pm7\, (\mathrm{mod}\, 25)$, then $\lambda\,=\,1-\zeta_5$ is ramified in $k$, so we get $d=3$.The condition  $n\, \not \equiv \, \pm1\pm7\, (\mathrm{mod}\, 25)$ is verified for all $q_1 \, \equiv \, 2\, (\mathrm{mod}\, 5),\,\, q_2 \, \equiv \, 3\, (\mathrm{mod}\, 5)$, if $e\,=\,2,3,4$ we have $n\,= \,5^{e}q_1^{2}q_2 \,\, \equiv \, 0\, (\mathrm{mod}\, 25)$, if $n\,=\, 5q_1^{2}q_2$ we have $q_1^2q_2\,=\,5\alpha+2$  with $\alpha, \in \mathbb{Z}$, so $5q_1^{2}q_2 \, \equiv \, 10\, (\mathrm{mod}\, 25)$. If $q_1 \, \equiv \, 7\, (\mathrm{mod}\, 25)$ and $q_2 \, \equiv \, -7\, (\mathrm{mod}\, 25)$ we have $q^*\,=\,2$, and if $q_1 \, \not\equiv \, 7\, (\mathrm{mod}\, 25)$ or $q_2 \, \not\equiv \, -7\, (\mathrm{mod}\, 25)$, according to the proof of [\ref{Mani}, theorem 5.13] we have $q^*\,=\,1$, but for this form of the radicand $n$ the computational number theory system \textbf{PARI} [\ref{PRI}] show that $C_{k,5}\,\simeq\,\mathbb{Z}/5\mathbb{Z}$.\\

\item[$(iii)$] $n\,=\, p^{e}q_1^{e_1} \, \equiv \, \pm1\pm7\, (\mathrm{mod}\, 25)$ with $p \, \equiv \, -1\, (\mathrm{mod}\, 5),\,\, q_1 \, \equiv \, \pm2\, (\mathrm{mod}\, 5)$ and $e,e_1 \in \{1,2,3,4\}$. As $\mathbb{Q}(\sqrt[5]{ab^2c^3d^4})\,=\,\mathbb{Q}(\sqrt[5]{a^2b^4cd^3})\,=\,\mathbb{Q}(\sqrt[5]{a^3bc^4d^2})\,=\,\mathbb{Q}(\sqrt[5]{a^4b^3c^2d})$ we can choose  $e_1\,=\,1$ i.e $n\,=\, p^{e}q_1$ with $e \in \{1,2,3,4\}$. By Corollary \ref{deco} $p\,=\,\pi_1\pi_2$ where $\pi_1,\,\pi_2$ are primes in $k_0$ and $q_2$ is inert in $k_0$. We have $p$ is ramified in $\Gamma$,so $\pi_1,\,\pi_2$ are ramified in $k$. $q_2$ is ramified in $\Gamma$ too, we obtain $d=3$. The condition $n\,=\, p^{e}q_1 \, \equiv \, \pm1\pm7\, (\mathrm{mod}\, 25)$ is not verified for all $p \, \equiv \, -1\, (\mathrm{mod}\, 5),\,\, q_1 \, \equiv \, \pm2\, (\mathrm{mod}\, 5)$, so we combine all the cases of congruence and we obtain that $p \, \not\equiv \,-1\, (\mathrm{mod}\, 25)$ and $ q_1 \, \not\equiv \,\pm7\, (\mathrm{mod}\, 25)$. According to [\ref{Mani}, theorem 5.18] If $p \, \equiv \, -1\, (\mathrm{mod}\, 25)$ and $q_1 \, \equiv \, \pm7\, (\mathrm{mod}\, 25)$ we have rank $C_{k,5} \geq 3$ which is impossible in our cases. If $p \, \not\equiv \,-1\, (\mathrm{mod}\, 25)$ and $q_1 \, \not\equiv \, \pm7\, (\mathrm{mod}\, 25)$, according to [\ref{Mani}, theorem 5.13] we have $q^*\,=\,1$. Using the computational number theory system \textbf{PARI/GP} [\ref{PRI}], if $n\,=\,p^{e}q_1$ with $p \, \not\equiv \, -1\, (\mathrm{mod}\, 25)$ and $ q_1 \, \not\equiv \, \pm7\, (\mathrm{mod}\, 25)$, the field $k$ is not always of type $(5,5)$.
\end{enumerate}

We summarize all forms of integer $n$ in the case 2, for which $k$ is of type $(5,5)$ and rank $(C_{k,5}^{(\sigma)})\,=\,1$ as follows:
\begin{equation}
\label{}
 n=\left\lbrace
   \begin{array}{ll}
   
   5^ep \not \equiv \, \pm1,\pm7\, (\mathrm{mod}\,25) & \text{ with } p\, \not\equiv \, -1\,(\mathrm{mod}\,25),\\

   p^{e}q_1 \, \equiv \, \pm1,\pm7\, (\mathrm{mod}\,25) & \text{ with } p\, \not\equiv \, -1\,(\mathrm{mod}\,25) \text{ and } q_1\, \not\equiv \, \pm7\,(\mathrm{mod}\,25)\\
   
   \end{array}
   \right.
\end{equation}
\item Case 3: we have $q^* = 2$ and $d=2$, so the number of prime ideals which are ramified in $k/k_0$ should be $2$. Let $n\,=\,\mu\lambda^{e}\pi_1^{e_1}....\pi_g^{e_g}$, where $\mu$ is a unit in $\mathcal{O}_{k_0}$, $\lambda = 1-\zeta_5$, $\pi_1,,,,\pi_g$ are primes in $k_0$ and $e \in \{0,1,2,3,4\}$, $e_i \in \{1,2,3,4\}$ for $1\leq i\leq g$.\\
$d\,=\,g$ or $g+1$ according to whether $n \, \equiv \, \pm1\pm7\, (\mathrm{mod}\, 25)$ or $n \not \equiv \, \pm1\pm7\, (\mathrm{mod}\, 25)$, to obtain $d=2$, $n$ must be written in $\mathcal{O}_{k_0}$ as: $n\,=\, \pi_1^{e_1}\pi_2^{e_2}$ or $n\,=\, \lambda^{e}\pi_1^{e_1}$, therefore we have three forms of $n$: 
\begin{enumerate}
\item[$(i)$] $n\,=\, 5^eq_1^{e_1} \not \equiv \,\pm1\pm7\, (\mathrm{mod}\, 25)$ with $q_1 \, \equiv \, \pm2\, (\mathrm{mod}\, 5)$ and $e,e_1 \in \{1,2,3,4\}$. As $\mathbb{Q}(\sqrt[5]{ab^2c^3d^4})\,=\,\mathbb{Q}(\sqrt[5]{a^2b^4cd^3})\,=\,\mathbb{Q}(\sqrt[5]{a^3bc^4d^2})\,=\,\mathbb{Q}(\sqrt[5]{a^4b^3c^2d})$ we can choose  $e_1\,=\,1$ i.e $n\,=\, 5^eq_1 $ with $e \in \{1,2,3,4\}$. Since $q^*=2$ so we have $q_1\,\, \equiv \,\,\pm7\, (\mathrm{mod}\, 25)$. By Corollary \ref{deco} $q_1$ is inert in $k_0$, we have $q_1$ is ramified in $\Gamma$, so $q_1$ is ramified too in $k$, and since $n\,\not \equiv \,\pm1\pm7\, (\mathrm{mod}\, 25)$, according to [\ref{Mani}, Lemma 5.1], $\lambda \,=\, 1-\zeta_5$ is ramified in $k$ so we obtain $d=2$. The condition $n\,\not \equiv \,\pm1\pm7\, (\mathrm{mod}\, 25)$ is verified for all $q_1 \,\, \equiv \,\, \pm7\, (\mathrm{mod}\, 5)$, because if $e\,=\,2,3,4$ we have $n\,=\,5^eq_1\, \equiv \, 0\, (\mathrm{mod}\, 25)$, if $e=1$ i.e $n\,=\, 5q_1$ we have $n\,=\,5q_1 \, \equiv \, \pm10\, (\mathrm{mod}\, 25)$, we conclude that $n\,=\,5^eq_1 \not \equiv \, \pm1\pm7\, (\mathrm{mod}\, 25)$ with $q_1 \, \equiv \, \pm7\, (\mathrm{mod}\, 25)$, but for this form of the radicand $n$ the computational number theory system \textbf{PARI/GP} [\ref{PRI}] show that $C_{k,5}\,\simeq\,\mathbb{Z}/5\mathbb{Z}$.

\item[$(ii)$] $n\,=\, q_1^{e_1}q_2^{e_2} \, \equiv \, \pm1\pm7\, (\mathrm{mod}\, 25)$ with $q_1 \, \equiv \, 3\, (\mathrm{mod}\, 5),\,\, q_2 \, \equiv \, 2\, (\mathrm{mod}\, 5)$ and $e_1,e_2 \in \{1,2,3,4\}$. As $\mathbb{Q}(\sqrt[5]{ab^2c^3d^4})\,=\,\mathbb{Q}(\sqrt[5]{a^2b^4cd^3})\,=\,\mathbb{Q}(\sqrt[5]{a^3bc^4d^2})\,=\,\mathbb{Q}(\sqrt[5]{a^4b^3c^2d})$ we can choose  $e_2\,=\,1$ i.e $n\,=\, q_1^{e_1}q_2$ with $e_1 \in \{1,2,3,4\}$. Since $q^*=2$, we have $\zeta_5\,\in\,N_{k/k_0}(k^*)$, we get that $q_1\,\, \equiv \,\,-7\, (\mathrm{mod}\, 25)$ and $q_2\,\, \equiv \,\,7\, (\mathrm{mod}\, 25)$.  By Corollary \ref{deco} $q_1$ and $q_2$ are inert in $k_0$, and we have $q_1,\,q_2$ are ramified in $\Gamma$,so $q_1,\,q_2$ are ramified in $k$, so we obtain $d=2$. The condition $n\,=\, q_1^{e_1}q_2 \, \equiv \, \pm1\pm7\, (\mathrm{mod}\, 25)$ is verified, because we have $n\,=\,q_1^{e_1}q_2 \, \equiv \, \pm7\, (\mathrm{mod}\, 25)$ for all $e_1$, so we conclude that $n\,=\, q_1^{e_1}q_2 \, \equiv \, \pm1\pm7\, (\mathrm{mod}\, 25)$ with $q_1 \, \equiv \, -7\, (\mathrm{mod}\, 25),\,\, q_2 \, \equiv \, 7\, (\mathrm{mod}\, 25)$, but for this form of the radicand $n$ the computational number theory system \textbf{PARI/GP} [\ref{PRI}] show that $C_{k,5}\,\simeq\,\mathbb{Z}/5\mathbb{Z}$

\item[$(iii)$] $n\,=\, p^{e} \, \equiv \, \pm1\pm7\, (\mathrm{mod}\, 25)$ with $p \, \equiv \, -1\, (\mathrm{mod}\, 5)$ and $e  \in \{1,2,3,4\}$. Since $q^*=2$, we have $\zeta_5\,\in\,N_{k/k_0}(k^*)$, we get that $p\,\, \equiv \,\,-1\, (\mathrm{mod}\, 25)$. By Corollary \ref{deco}, $p\,=\,\pi_1\pi_2$ where $\pi_1,\,\pi_2$ are primes of $k_0$. The prime $p$ is ramified in $\Gamma$, then $\pi_1,\pi_2$ are ramified in $k$. hence we have $d=2$. The condition $n\,=\, p^{e} \, \equiv \, \pm1\pm7\, (\mathrm{mod}\, 25)$ is verified for all $p\,\, \equiv \,\,-1\, (\mathrm{mod}\, 25)$, we conclude that $n\,=\, p^{e} \, \equiv \, \pm1\pm7\, (\mathrm{mod}\, 25)$ with $p \, \equiv \, -1\, (\mathrm{mod}\, 25)$. Using the computational number theory system \textbf{PARI/GP} [\ref{PRI}], if $n\,=\, p^{e} \, \equiv \, \pm1\pm7\, (\mathrm{mod}\, 25)$ with $p \, \equiv \, -1\, (\mathrm{mod}\, 25)$, the field $k$ is not always of type $(5,5)$.

\end{enumerate}
We deduce that in the case 3, there is one form of $n$ for which the fields $k$ is of type $(5,5)$ and rank$(C_{k,5}^{(\sigma)})\,=\,1$ as follows:
\begin{center}
 $n\,=\,p^{e} \, \equiv \, \pm1,\pm7\, (\mathrm{mod}\,25)  \text{ with } p\, \equiv \, -1\,(\mathrm{mod}\,25)$\\
    
\end{center}
\end{itemize}

(2) If rank $(C_{k,5}^{(\sigma)})\,=\,2$, so  $C_{k,5}\,=\,C_{k,5}^{(\sigma)}$. According to [\ref{Mani},section 5.1], the rank of $C_{k,5}^{(\sigma)}$ is given as follows:
\begin{center}
rank $(C_{k,5}^{(\sigma)})\,=\, d-3+q^*$
\end{center}
where $d$ et $q^*$ are defined previously. Since rank $(C_{k,5}^{(\sigma)})\, =\, 2$ and $q^*\,=\,0, 1\, \mathrm{or}\, 2$, there are three possible cases as follows:\\
\begin{itemize}
    \item Case 1: \(q^*=0\,\, \mathrm{and}\,\, d=5\),
    \item Case 2: \(q^*=1\,\, \mathrm{and}\,\, d=4\),
    \item Case 3: \(q^*=2\,\, \mathrm{and}\,\, d=3\), \end{itemize}
We will treat the three cases to prove the forms of the radicand $n$. By theorem \ref{ran2}, $n$ must be divisible by one prime $l \, \equiv \, 1\,(\mathrm{mod}\,5)$ in all cases, and since rank $(C_{k,5}^{(\sigma)})\, =\, 2$, the invariant $q^*$ should be $0$ or $1$, because if $q^*\,=\,2$ and $l\,\equiv\,1\,(\mathrm{mod}\,5)$ divides $n$, we get that the invariant $d$ is at least $4$, so we obtain that rank $(C_{k,5}^{(\sigma)})$ is at least $3$.
\begin{itemize}
\item Case 1: we have $q^* = 0$ and $d=5$, so the number of prime ideals which are ramified in $k/k_0$ should be $5$. The radicand $n$ must be divisible by one prime $l \, \equiv \, 1\, (\mathrm{mod}\, 5)$.\\
We can writ $n\,\in\,\mathcal{O}_{k_0}$ as $n\,=\,\mu\lambda^{e}\pi_1^{e_1}....\pi_g^{e_g}$, where $\mu$ is a unit in $\mathcal{O}_{k_0}$, $\lambda = 1-\zeta_5$, $\pi_1,,,,\pi_g$ are primes in $k_0$ and $e \in \{0,1,2,3,4\}$, $e_i \in \{1,2,3,4\}$ for $1\leq i\leq g$.\\
$d\,=\,g$ or $g+1$ according to whether $n \, \equiv \, \pm1\pm7\, (\mathrm{mod}\, 25)$ or $n \not \equiv \, \pm1\pm7\, (\mathrm{mod}\, 25)$. To obtain $d=5$, $n$ must be written in $\mathcal{O}_{k_0}$ as: $n\,=\, \pi_1^{e_1}\pi_2^{e_2}\pi_3^{e_3}\pi_4^{e_4}\pi_5^{e_5}$ or $n\,=\, \lambda^{e}\pi_1^{e_1}\pi_2^{e_2}\pi_3^{e_3}\pi_4^{e_4}$, therefore we have two forms of $n$: 
\begin{enumerate}
\item[$(i)$] $n\,=\, 5^el^{e_1} \not \equiv \,\pm1\pm7\, (\mathrm{mod}\, 25)$ with $l \, \equiv \, 1\, (\mathrm{mod}\, 5)$ and $e,e_1 \in \{1,2,3,4\}$. As $\mathbb{Q}(\sqrt[5]{ab^2c^3d^4})\,=\,\mathbb{Q}(\sqrt[5]{a^2b^4cd^3})\,=\,\mathbb{Q}(\sqrt[5]{a^3bc^4d^2})\,=\,\mathbb{Q}(\sqrt[5]{a^4b^3c^2d})$ we can choose  $e_1\,=\,1$ i.e $n\,=\, 5^el $ with $e \in \{1,2,3,4\}$. By Corollary \ref{deco} $l\,=\, \pi_1\pi_2\pi_3\pi_4$ with $\pi_i$ are primes in $k_0$, we have $l$ is ramified in $\Gamma$, so $\pi_1,\,\pi_2,\pi_3$ and $\pi_4$ are ramified in $k$, and since $n\,\not \equiv \,\pm1\pm7\, (\mathrm{mod}\, 25)$, according to [\ref{Mani}, Lemma 5.1], $\lambda \,=\, 1-\zeta_5$ is ramified in $k$ so we obtain $d=5$. The condition $n\,\not \equiv \,\pm1\pm7\, (\mathrm{mod}\, 25)$ is verified for all $l \, \equiv \, 1\, (\mathrm{mod}\, 5)$, because if $e\,=\,2,3,4$ we have $n\,=\,5^el\, \equiv \, 0\, (\mathrm{mod}\, 25)$, if $e=1$ i.e $n\,=\, 5l$ we have $l\,=\,5\alpha+1$ that implie $5l\,=\,25\alpha+5$ with $\alpha \in \mathbb{Z}$, so $n\,=\,5l \, \equiv \, 5\, (\mathrm{mod}\, 25)$. If $l \, \equiv \, 1 (\mathrm{mod}\, 25)$ we have $\zeta_5\,\in\,N_{k/k_0}(k^*)$, so $q^*\,\geq\,1$ which in impossible in this case. We conclude that $n\,=\,5^el \not \equiv \, \pm1\pm7\, (\mathrm{mod}\, 25)$ with $l \, \not\equiv \, 1 (\mathrm{mod}\, 25)$. Using the computational number theory system \textbf{PARI/GP} [\ref{PRI}], if $n\,=\,5^el \not\equiv \, \pm1\pm7\, (\mathrm{mod}\, 25)$ with $l \, \not\equiv \, 1 (\mathrm{mod}\, 25)$, the field $k$ is not always of type $(5,5)$.\\

\item[$(ii)$] $n\,=\, l^{e}q_1^{e_1} \, \equiv \, \pm1\pm7\, (\mathrm{mod}\, 25)$ with $l \, \equiv \, 1\, (\mathrm{mod}\, 5),\,\, q_1 \, \equiv \, \pm2\, (\mathrm{mod}\, 5)$ and $e, e_1 \in \{1,2,3,4\}$. As $\mathbb{Q}(\sqrt[5]{ab^2c^3d^4})\,=\,\mathbb{Q}(\sqrt[5]{a^2b^4cd^3})\,=\,\mathbb{Q}(\sqrt[5]{a^3bc^4d^2})\,=\,\mathbb{Q}(\sqrt[5]{a^4b^3c^2d})$ we can choose  $e_1\,=\,1$ i.e $n\,=\, l^{e}q_1$ with $e \in \{1,2,3,4\}$. By Corollary \ref{deco} $l\,=\,\pi_1\pi_2\pi_3\pi_4$ where $\pi_i$ are primes in $k_0$ and $q_1$ is inert in $k_0$. We know that $l$ is ramified in $\Gamma$,so $\pi_1,\,\pi_2,\pi_3$ and $\pi_4$ are ramified in $k$. $q_1$ is ramified in $\Gamma$ too, we obtain $d=5$. The condition $n\,=\, l^{e}q_1 \, \equiv \, \pm1\pm7\, (\mathrm{mod}\, 25)$ is not verified for all $l \, \equiv \, 1\, (\mathrm{mod}\, 5),\,\, q_1 \, \equiv \, \pm2\, (\mathrm{mod}\, 5)$, so we combine all the cases of congruence and we obtain that $l \, \equiv \, 1\, (\mathrm{mod}\, 5)$ and $ q_1 \, \equiv \, \pm2,\pm3\,\pm7 \,(\mathrm{mod}\, 25)$. Using the computational number theory system \textbf{PARI/GP} [\ref{PRI}], if $n\,=\, l^{e}q_1 \, \equiv \, \pm1\pm7\, (\mathrm{mod}\, 25)$, with $l \, \equiv \,1\, (\mathrm{mod}\, 5)$, and $ q_1 \, \equiv \, \pm2,\pm3\,\pm7\, (\mathrm{mod}\, 25)$, the field $k$ is not always of type $(5,5)$
\end{enumerate}

We summarize all forms of integer $n$ in this case  as follows:
\begin{equation}
\label{}
 n=\left\lbrace
   \begin{array}{ll}
   
   5^el \not \equiv \, \pm1,\pm7\, (\mathrm{mod}\,25) & \text{ with } l\, \not\equiv \, 1\,(\mathrm{mod}\,25),\\

   l^{e}q_1 \, \equiv \, \pm1,\pm7\, (\mathrm{mod}\,25) & \text{ with } l \, \equiv \, 1\, (\mathrm{mod}\, 5)\, q_1 \, \equiv \, \pm2,\pm3\,\pm7\, (\mathrm{mod}\, 25)\\
   
   \end{array}
   \right.
\end{equation}

\item Case 2: We have $q^*\,=\,1$ and $d=4$, so the number of prime ideals which are ramified in $k/k_0$ should be $4$. The radicand $n$ must be divisible by one prime $l \, \equiv \, 1\, (\mathrm{mod}\, 5)$, and according to Corollary \ref{deco} $l$ splits in $k_0$ as $l\,=\, \pi_1\pi_2\pi_3\pi_4$ with $\pi_i$ are primes in $k_0$. Since $l$ is ramified in $\Gamma$, so $\pi_1,\,\pi_2,\pi_3$ and $\pi_4$ are ramified in $k$, hence if $n$ is devisible by another prime than $l$ the number of primes which are ramified in $k/k_0$ surpass $4$, therefore we have unique form of $n$ in this case, its $n\,=\, l^{e} \, \equiv \, \pm1\pm7\, (\mathrm{mod}\, 25)$ with $l \, \equiv \, 1\, (\mathrm{mod}\, 5)$ and $e\in \{1,2,3,4\}$. The condition $n\, \equiv \,\, \pm1\pm7\, (\mathrm{mod}\, 25)$ is verified only for $l\,  \, \equiv \, 1, (\mathrm{mod}\, 25)$, and we have $q^*\,=\,1$. In conclusion we get $n\,=\, l^{e}$ with $l\, \, \equiv \, 1, (\mathrm{mod}\, 25)$. Using the computational number theory system \textbf{PARI/GP} [\ref{PRI}], if $n\,=\,l^{e} \equiv \, \pm1\pm7\, (\mathrm{mod}\, 25)$ with $l \, \equiv \, 1 (\mathrm{mod}\, 25)$, the field $k$ is not always of type $(5,5)$

\item case 3: We have $q^*\,=\,2$ and $d=3$, so the number of prime ideals which are ramified in $k/k_0$ should be $3$. The radicand $n$ must be divisible by one prime $l \, \equiv \, 1\, (\mathrm{mod}\, 5)$, and according to Corollary \ref{deco} $l\,=\, \pi_1\pi_2\pi_3\pi_4$ with $\pi_i$ are primes in $k_0$. Since $p$ is ramified in $\Gamma$, so $\pi_1,\,\pi_2,\pi_3$ and $\pi_4$ are ramified in $k$, so we deduce that the number of primes ramified in $k/k_0$ is at least $4$, so the case 3 does not exist.
\end{itemize}

\section{Numerical examples}
Let $\Gamma\,=\,\mathbb{Q}(\sqrt[5]{n})$ be a pure quintic field, where $n$ is a positive integer, $5^{th}$ power-free, and let $k\,=\,\mathbb{Q}(\sqrt[5]{n},\zeta_5)$ its normal closure. We assume that $C_{k,5}$ is of type $(5,5)$. Using the system \textbf{PARI/GP} [\ref{PRI}], we illustrate our main result Theorem \ref{thm:Rank1}. 
\subsection{rank $(C_{k,5}^{(\sigma)})\, =\, 1$}
\begin{center}
Table 1: $n\,=\, p^{e}\,\equiv \, \pm1\pm7\, (\mathrm{mod}\, 25)$ with $p\,\equiv \,-1\, (\mathrm{mod}\, 25)$\\
\begin{tabular}{|c|c|c|c|c|c|c|}
\hline 
$p$ &  $n\,=\,p^{e}$ & $p\,(\mathrm{mod}\,5)$ & $p\,(\mathrm{mod}\,25)$ & $h_{k,5}$ & $C_{k,5}$ & rank $(C_{k,5}^{(\sigma)})$ \\ 
\hline 
149 & 22201 = $149^2$ & -1 & -1 & 25 & $(5,5)$ & 1 \\ 
199 & 7880599 = $199^3$ & -1 & -1 & 25 & $(5,5)$ & 1 \\ 
349 & 42508549 = $349^3$ & -1 & -1 & 25 & $(5,5)$ & 1 \\ 
449 &  $449$ & -1 & -1 & 25 & $(5,5)$ & 1 \\ 
559 &  $559$ & -1 & -1 & 25 & $(5,5)$ & 1 \\ 
1249 &  $1249$ & -1 & -1 & 25 & $(5,5)$ & 1 \\ 
1499 &  $1499$ & -1 & -1 & 25 & $(5,5)$ & 1 \\ 
1949 &  $1949$ & -1 & -1 & 25 & $(5,5)$ & 1 \\ 
1999 &  $1999$ & -1 & -1 & 25 & $(5,5)$ & 1 \\ 
2099 &  $449$ & -1 & -1 & 25 & $(5,5)$ & 1 \\ 
\hline 
\end{tabular} 
\end{center}
\newpage
\hspace{4cm} Table 2: $n\,=\, q_1^{e_1}q_2\,\equiv \, \pm1\pm7\, (\mathrm{mod}\, 25)$ with $q_i\,\equiv \,\pm7\, (\mathrm{mod}\, 25)$\\
\begin{tabular}{|c|c|c|c|c|c|c|c|c|c|}
\hline 
$q_1$ & \small$q_1\,(\mathrm{mod}\,5)$ & \small$q_1\,(\mathrm{mod}\,25)$& $q_2$ & \small$q_2\,(\mathrm{mod}\,5)$ & \small$q_2\,(\mathrm{mod}\,25)$& $n\,=\,q_1^{e_1}q_2$ & $h_{k,5}$ & $C_{k,5}$ & \small rank $(C_{k,5}^{(\sigma)})$\\ 
\hline 
7 & 2 & 7 & 43 & 3 & -7 & 2107 = $7^2\times43$ & 5 & $\mathbb{Z}/5\mathbb{Z}$ & 1 \\
7 & 2 & 7 & 193 & 3 & -7 & 1351 = $7\times193$ & 5 & $\mathbb{Z}/5\mathbb{Z}$ & 1 \\ 
7 & 2 & 7 & 293 & 3 & -7 & 2051 = $7\times293$ & 5 & $\mathbb{Z}/5\mathbb{Z}$ & 1 \\ 
107 & 2 & 7 & 43 & 3 & -7 & 492307 = $107^2\times43$ & 5 & $\mathbb{Z}/5\mathbb{Z}$ & 1 \\
107 & 2 & 7 & 193 & 3 & -7 & 20651 = $107\times193$ & 5 & $\mathbb{Z}/5\mathbb{Z}$ & 1 \\
107 & 2 & 7 & 293 & 3 & -7 & 31351 = $107\times293$ & 5 & $\mathbb{Z}/5\mathbb{Z}$ & 1 \\
107 & 2 & 7 & 443 & 3 & -7 & 47401 = $107\times443$ & 5 & $\mathbb{Z}/5\mathbb{Z}$ & 1 \\
157 & 2 & 7 & 43 & 3 & -7 & 6751 = $157\times43$ & 5 & $\mathbb{Z}/5\mathbb{Z}$ & 1 \\
157 & 2 & 7 & 193 & 3 & -7 & 30301 = $157\times193$ & 5 & $\mathbb{Z}/5\mathbb{Z}$ & 1 \\
157 & 2 & 7 & 443 & 3 & -7 & 69551 = $157\times443$ & 5 & $\mathbb{Z}/5\mathbb{Z}$ & 1 \\
257 & 2 & 7 & 193 & 3 & -7 & 49601 = $257\times293$ & 5 & $\mathbb{Z}/5\mathbb{Z}$ & 1 \\
257 & 2 & 7 & 293 & 3 & -7 & 75301 = $257\times293$ & 5 & $\mathbb{Z}/5\mathbb{Z}$ & 1 \\
307 & 2 & 7 & 193 & 3 & -7 & 59251 = $307\times193$ & 5 & $\mathbb{Z}/5\mathbb{Z}$ & 1 \\
457 & 2 & 7 & 43 & 3 & -7 & 19651 = $457\times43$ & 5 & $\mathbb{Z}/5\mathbb{Z}$ & 1 \\
457 & 2 & 7 & 443 & 3 & -7 & 202451 = $457\times443$ & 5 & $\mathbb{Z}/5\mathbb{Z}$ & 1 \\
557 & 2 & 7 & 43 & 3 & -7 & 23251 = $557\times43$ & 5 & $\mathbb{Z}/5\mathbb{Z}$ & 1 \\
607 & 2 & 7 & 43 & 3 & -7 & 26101 = $607\times43$ & 5 & $\mathbb{Z}/5\mathbb{Z}$ & 1 \\
\hline 
\end{tabular} 

\begin{center}
 Table 3 : $n\,=\, 5^{e}q_1 \not\equiv \, \pm1\pm7\, (\mathrm{mod}\, 25)$ with $q_1\,\equiv \,\pm7\, (\mathrm{mod}\, 25)$\\
\begin{tabular}{|c|c|c|c|c|c|c|}
\hline 
$q_1$ &  $n\,=\,5^eq_1$ & $q_1\,(\mathrm{mod}\,5)$ & $q_1\,(\mathrm{mod}\,25)$ & $h_{k,5}$ & $C_{k,5}$ & rank $(C_{k,5}^{(\sigma)})$ \\ 
\hline 
7 & 175 = $5^2\times7$ & 2 & 7 & 5 & $\mathbb{Z}/5\mathbb{Z}$ & 1 \\ 
107 & 535 = $5\times107$ & 2 & 7 & 5 & $\mathbb{Z}/5\mathbb{Z}$ & 1 \\ 
157 & 19625 = $5^3\times157$ & 2 & 7 & 5 & $\mathbb{Z}/5\mathbb{Z}$ & 1 \\  
257 & 6425 = $5^2\times257$ & 2 & 7 & 5 & $\mathbb{Z}/5\mathbb{Z}$ & 1 \\ 
307 & 38375 = $5^3\times307$ & 2 & 7 & 5 & $\mathbb{Z}/5\mathbb{Z}$ & 1 \\ 
457 & 2285 = $5\times457$ & 2 & 7 & 5 & $\mathbb{Z}/5\mathbb{Z}$ & 1 \\ 
557 & 2785 = $5\times557$ & 2 & 7 & 5 & $\mathbb{Z}/5\mathbb{Z}$ & 1 \\ 
607 & 3053 = $5\times607$ & 2 & 7 & 5 & $\mathbb{Z}/5\mathbb{Z}$ & 1 \\ 
757 & 3785 = $5\times457$ & 2 & 7 & 5 & $\mathbb{Z}/5\mathbb{Z}$ & 1 \\ 
857 & 4285 = $5\times457$ & 2 & 7 & 5 & $\mathbb{Z}/5\mathbb{Z}$ & 1 \\ 
907 & 4535 = $5\times907$ & 2 & 7 & 5 & $\mathbb{Z}/5\mathbb{Z}$ & 1 \\ 
43 & 1075 = $5^2\times43$ &3 & -7 & 5 & $\mathbb{Z}/5\mathbb{Z}$ & 1 \\ 
193 & 120625 = $5^4\times193$ & 3 & -7 & 5 & $\mathbb{Z}/5\mathbb{Z}$ & 1 \\
293 & 120625 = $5^4\times193$ & 3 & -7 & 5 & $\mathbb{Z}/5\mathbb{Z}$ & 1 \\ 
443 & 11075 = $5^2\times443$ & 3 & -7 & 5 & $\mathbb{Z}/5\mathbb{Z}$ & 1 \\ 
643 & 3215 = $5\times643$ & 3 & -7 & 5 & $\mathbb{Z}/5\mathbb{Z}$ & 1 \\ 
\hline
\end{tabular}
\end{center}
\newpage
\begin{center}
Table 4: $n\,=\, 5^{e}q_1^{2}q_2 \not \equiv \, \pm1\pm7\, (\mathrm{mod}\, 25)$ with $q_1$ or $q_2$  $\not\equiv \,\pm7\, (\mathrm{mod}\, 25)$\\
\begin{tabular}{|c|c|c|c|c|c|c|c|c|c|}
\hline 
$q_1$ & \small$q_1\,(\mathrm{mod}\,5)$ & \small$q_1\,(\mathrm{mod}\,25)$& $q_2$ & \small$q_2\,(\mathrm{mod}\,5)$ & \small$q_2\,(\mathrm{mod}\,25)$& $n\,=\,5^eq_1^2q_2$ & $h_{k,5}$ & $C_{k,5}$ & rank $(C_{k,5}^{(\sigma)})$\\ 
\hline 
2 & 2 & 2 & 3 & 3 & 3 & 60 = $5\times2^2\times3$ & 5 & $\mathbb{Z}/5\mathbb{Z}$ & 1 \\
2 & 2 & 2 & 13 & 3 & 13 & 260 = $5\times2^2\times13$ & 5 & $\mathbb{Z}/5\mathbb{Z}$ & 1 \\
2 & 2 & 2 & 53 & 3 & 3 & 1060 = $5\times2^2\times53$ & 5 & $\mathbb{Z}/5\mathbb{Z}$ & 1 \\
2 & 2 & 2 & 23 & 3 & -2 & 460 = $5\times2^2\times23$ & 5 & $\mathbb{Z}/5\mathbb{Z}$ & 1 \\
7 & 2 & 7 & 3 & 3 & 3 & 735 = $5\times7^2\times3$ & 5 & $\mathbb{Z}/5\mathbb{Z}$ & 1 \\ 
17 & 2 & 17 & 3 & 3 & 3 & 108375 = $5^3\times17^2\times3$ & 5 & $\mathbb{Z}/5\mathbb{Z}$ & 1 \\
17 & 2 & 17 & 23 & 3 & -2 & 33235 = $5\times17^2\times23$ & 5 & $\mathbb{Z}/5\mathbb{Z}$ & 1\\
37 & 2 & 17 & 3 & 3 & 3 & 20535 = $5\times37^2\times13$ & 5 & $\mathbb{Z}/5\mathbb{Z}$ & 1 \\
37 & 2 & 17 & 13 & 3 & 13 & 88985 = $5\times37^2\times13$ & 5 & $\mathbb{Z}/5\mathbb{Z}$ & 1 \\
47 & 2 & -3 & 3 & 3 & 3 & 33135 = $5\times47^2\times3$ & 5 & $\mathbb{Z}/5\mathbb{Z}$ & 1 \\
47 & 2 & -3 & 13 & 3 & 13 & 143585 = $5\times47^2\times13$ & 5 & $\mathbb{Z}/5\mathbb{Z}$ & 1 \\
47 & 2 & -3 & 23 & 3 & -2 & 254035 = $5\times47^2\times23$ & 5 & $\mathbb{Z}/5\mathbb{Z}$ & 1 \\
47 & 2 & -3 & 43 & 3 & -7 & 474935 = $5\times47^2\times43$ & 5 & $\mathbb{Z}/5\mathbb{Z}$ & 1 \\
107 & 2 & 7 & 23 & 3 & -2 & 1316635 = $5\times2^2\times3$ & 5 & $\mathbb{Z}/5\mathbb{Z}$ & 1 \\
67 & 2 & 17 & 3 & 3 & 3 & 67335 = $5\times67^2\times3$ & 5 & $\mathbb{Z}/5\mathbb{Z}$ & 1 \\
67 & 2 & 17 & 53 & 3 & 3 & 1189585 = $5\times67^2\times53$ & 5 & $\mathbb{Z}/5\mathbb{Z}$ & 1 \\
97 & 2 & -3 & 43 & 3 & -7 & 2022935 = $5\times97^2\times43$ & 5 & $\mathbb{Z}/5\mathbb{Z}$ & 1 \\
\hline 
\end{tabular}
\end{center}
\vspace{1cm}
\begin{center}
Table 5: $n\,=\, p^{e}q_1 \not \equiv \, \pm1\pm7\, (\mathrm{mod}\, 25)$ with $p\,\not\equiv \,-1\, (\mathrm{mod}\, 25)$, $q_1\,\not\equiv \,\pm7\, (\mathrm{mod}\, 25)$ 
\begin{tabular}{|c|c|c|c|c|c|c|c|c|c|}
\hline 
$p$ & \small$p\,(\mathrm{mod}\,5)$ & \small$p\,(\mathrm{mod}\,25)$& $q_1$ & \small$q_1\,(\mathrm{mod}\,5)$ & \small$q_1\,(\mathrm{mod}\,25)$& $n\,=\,p^{e}q_1$ & $h_{k,5}$ & $C_{k,5}$ & rank $(C_{k,5}^{(\sigma)})$\\ 
\hline 
59 & -1 & 9 & 2 & 2 & 2 & 118 = $59\times2$ & 25 & $(5,5)$ & 1 \\
19 & -1 & 19 & 3 & 3 & 3 & 57 = $19\times3$ & 25 & $(5,5)$ & 1 \\
59 & -1 & 9 & 23 & 3 & -2 & 1357 = $59\times23$ & 25 & $(5,5)$ & 1 \\
359 & -1 & 9 & 2 & 2 & 2 & 718 = $359\times2$ & 25 & $(5,5)$ & 1 \\
409 & -1 & 9 & 2 & 2 & 2 & 816 = $409\times2$ & 25 & $(5,5)$ & 1 \\
59 & -1 & 9 & 127 & 2 & 2 &  7493 = $59\times127$ & 25 & $(5,5)$ & 1 \\
109 & -1 & 9 & 23 & 3 & -2 & 2507 = $109\times23$ & 25 & $(5,5)$ & 1 \\
509 & -1 & 9 & 2 & 2 & 2 & 1018 = $509\times2$ & 25 & $(5,5)$ & 1 \\
709 & -1 & 9 & 2 & 2 & 2 & 1418 = $709\times2$ & 25 & $(5,5)$ & 1 \\
19 & -1 & 19 & 53 & 3 & 3 & 1007 = $19\times53$ & 25 & $(5,5)$ & 1 \\
\hline 
\end{tabular}
\end{center}
\newpage

\begin{center}
Table 6: $n\,=\,5^ep \not\equiv \, \pm1\pm7\, (\mathrm{mod}\, 25)$ with $p \, \not\equiv \, -1 (\mathrm{mod}\, 25)$\\
\hspace{1,5cm}
\begin{tabular}{|c|c|c|c|c|c|c|}
\hline 
$p$ &  $n\,=\,5^ep$ & $p\,(\mathrm{mod}\,5)$ & $p\,(\mathrm{mod}\,25)$ & $h_{k,5}$ & $C_{k,5}$ & rank $(C_{k,5}^{(\sigma)})$ \\ 
\hline 
19 & 475 = $5^2\times19$ & -1 & 19 & 25 & (5,5) & 1 \\ 
29 & 145 = $5\times29$ & -1 & 4 & 25 & (5,5) & 1 \\ 
59 & 7375 = $5^3\times59$ & -1 & 9 & 25 & (5,5) & 1 \\ 
89 & 55625 = $5^4\times89$ & -1 & 14 & 25 & (5,5) & 1 \\ 
109 & 2725 = $5^2\times109$ & -1 & 9 & 25 & (5,5) & 1 \\ 
229 & 28625 = $5^3\times229$ & -1 & 4 & 25 & (5,5) & 1 \\
239 & 1195 = $5\times239$ & -1 & 14 & 25 & (5,5) & 1 \\ 
269 & 6725 = $5^2\times19$ & -1 & 19 & 25 & (5,5) & 1 \\ 
379 & 168125 = $5^4\times379$ & -1 & 4 & 25 & (5,5) & 1 \\ 
389 & 1945 = $5^2\times389$ & -1 & 14 & 25 & (5,5) & 1 \\ 
\hline
\end{tabular}
\end{center}

\subsection{rank $(C_{k,5}^{(\sigma)})\, =\, 2$}
\begin{center}
Table 1: $n\,=\,5^el \not\equiv \, \pm1\pm7\, (\mathrm{mod}\, 25)$ with $l \, \not\equiv \, 1 (\mathrm{mod}\, 25)$\\
\begin{tabular}{|c|c|c|c|c|c|c|}
\hline 
$l$ &  $n\,=\,5^el$ & $l\,(\mathrm{mod}\,5)$ & $l\,(\mathrm{mod}\,25)$ & $h_{k,5}$ & $C_{k,5}$ & rank $(C_{k,5}^{(\sigma)})$ \\ 
\hline 
11 & 55 = $5\times11$ & 1 & 11 & 25 & (5,5) & 2 \\ 
41 &   5125 = $5^3\times41$ & 1 & -9 & 25 & (5,5) & 2 \\ 
61 & 5125 = $5^4\times61$ & 1 & 11 & 25 & (5,5) & 2 \\ 
71 & 1775 = $5^2\times71$ & 1 & -4 & 25 & (5,5) & 2 \\ 
131 & 655 = $5\times131$ & 1 & 6 & 25 & (5,5) & 2 \\ 
181 & 113125 = $5^4\times181$ & 1 & 6 & 25 & (5,5) & 2 \\
241 & 30125 = $5^3\times241$ & 1 & -9 & 25 & (5,5) & 2 \\ 
311 & 1555 = $5\times311$ & 1 & 11 & 25 & (5,5) & 2 \\ 
331 & 8275 = $5^2\times331$ & 1 & 6 & 25 & (5,5) & 2 \\ 
431 & 2155 = $5\times431$ & 1 & 6 & 25 & (5,5) & 2 \\ 
\hline 
\end{tabular}
\end{center}

\begin{center}
Table 2: $n\,=\, l^{e}q_1\,\equiv \, \pm1\pm7\, (\mathrm{mod}\, 25)$ with $l \, \equiv \,1\, (\mathrm{mod}\, 5)$ and $ q_1 \, \equiv \, \pm2,\pm3\,\pm7\, (\mathrm{mod}\, 25)$
\begin{tabular}{|c|c|c|c|c|c|c|c|c|c|}
\hline 
$l$ & \small$l\,(\mathrm{mod}\,5)$ & \small$l\,(\mathrm{mod}\,25)$& $q_1$ & \small$q_1\,(\mathrm{mod}\,5)$ & \small$q_1\,(\mathrm{mod}\,25)$& $n\,=\,l^{e}q_1$ & $h_{k,5}$ & $C_{k,5}$ & \small rank $(C_{k,5}^{(\sigma)})$\\ 
\hline 
31 & 1 & 6 & 2 & 2 & 2 &  $31\times2$ & 25 & $(5,5)$ & 2 \\
131 & 1 & 6 & 23 & 3 & -2 &  $131^3\times23$ & 25 & $(5,5)$ & 2 \\
181 & 1 & 6 & 47 & 2 & -3 &  $181\times47$ & 25 & $(5,5)$ & 2 \\
11 & 1 & 11 & 3 & 3 & 3 &  $11\times3$ & 25 & $(5,5)$ & 2 \\
41 & 1 & 16 & 23 & 3 & -2 &  $41\times23$ & 25 & $(5,5)$ & 2 \\
191 & 1 & 16 & 2 & 2 & 2 &  $191\times2$ & 25 & $(5,5)$ & 2 \\
41 & 1 & 16 & 47 & 2 & -3 &  $41^2\times47$ & 25 & $(5,5)$ & 2 \\
311 & 1 & 11 & 2 & 2 & 2 &  $311^4\times2$ & 25 & $(5,5)$ & 2 \\
\hline 
\end{tabular} 
\end{center}
\newpage
\begin{center}
Table 3: $n\,=\,l^{e} \equiv \, \pm1\pm7\, (\mathrm{mod}\, 25)$ with $l \, \equiv \, 1 (\mathrm{mod}\, 25)$ 
\begin{tabular}{|c|c|c|c|c|c|c|}
\hline 
$l$ &  $n\,=\,l^{e}$ & $l\,(\mathrm{mod}\,5)$ & $l\,(\mathrm{mod}\,25)$ & $h_{k,5}$ & $C_{k,5}$ & rank $(C_{k,5}^{(\sigma)})$ \\ 
\hline 
151 & $151$  & 1 & 1 & 25 & (5,5) & 2 \\ 
251 &   $251^2$  & 1 & 1 & 25 & (5,5) & 2 \\ 
601 & $601^3$ & 1 & 1 & 25 & (5,5) & 2 \\ 
1051 & $1051^4$ & 1 & 1 & 25 & (5,5) & 2 \\ 
1301 & $1301$ & 1 & 1 & 25 & (5,5) & 2 \\ 
1451 & $1451^2$ & 1 & 1 & 25 & (5,5) & 2 \\
1801 & $1801^3$ & 1 & 1 & 25 & (5,5) & 2 \\ 
1901 & $1901^4$ & 1 & 1 & 25 & (5,5) & 2 \\ 
2111 & 2111 & 1 & 1 & 25 & (5,5) & 2 \\ 
2131 & $2131^2$ & 1 & 1 & 25 & (5,5) & 2 \\ 
\hline 
\end{tabular} 
\end{center}

\section{Conjecture}
In this article, we have classified some pure quintic fields $\mathbb{Q}(\sqrt[5]{n})$, more precisely, we focused on the ones whose normal closures $\mathbb{Q}(\sqrt[5]{n},\zeta_5)$ possesses a $5$-class groups of type $(5,5)$, by treating the rank of the ambiguous classes, that can be characterized by the radicand $n$.\\
As to provide numerical examples, we use the system \textbf{PARI/GP} \ref{PRI}. Thus, we have noticed that the done calculations for some $n$ forms, show that $5$-class groupe $C_{k,5}$ of the field $k$, is isomorphic to $\mathbb{Z}/5\mathbb{Z}$, which allows us to give this conjecture as follows:\\
\begin{conjecture}
Let $\Gamma\,=\,\mathbb{Q}(\sqrt[5]{n})$ be a pure quintic field, where $n$ is a positive integer, $5^{th}$ power-free. Let $k\,= \Gamma(\zeta_5)$ be the normal closure of $\Gamma$. Denote by $C_{k,5}$ the 5-class group of $k$, $q_1,q_2\,\equiv \, \pm2\, (\mathrm{mod}\, 5)$ are primes and $e \in \{1,2,3,4\}$.\\
 If the radicand $n$ take one form as follows:
\begin{equation}
 n\,=\,
   \begin{cases}
  
  q_1^{e_1}q_2\,\equiv \, \pm1\pm7\, (\mathrm{mod}\, 25) & \text{ with } \quad q_i\,\equiv \,\pm7\, (\mathrm{mod}\, 25)\\
   
   5^{e}q_1 \not\equiv \, \pm1\pm7\, (\mathrm{mod}\, 25)& \text{ with } \quad  q_1\,\equiv \,\pm7\, (\mathrm{mod}\, 25)\\
   
  5^{e}q_1^{2}q_2 \not \equiv \, \pm1\pm7\, (\mathrm{mod}\, 25) & \text{ with } \quad q_1 \,\text{ or }\, q_2$  $\not\equiv \,\pm7\, (\mathrm{mod}\, 25)\\
   
   \end{cases}
\end{equation}
Then $C_{k,5}$ is a cyclic groupe of order $5$.
\end{conjecture}


Fouad ELMOUHIB\\
Department of Mathematics and Computer Sciences,\\
Mohammed 1st University,\\
Oujda - Morocco,\\
fouad.cd@gmail.com.\\\\

Mohamed TALBI\\
Regional Center of Professions of Education and Training in the Oriental,\\
Oujda - Morocco,\\
ksirat1971@gmail.com.\\\\

Abdelmalek AZIZI\\
Department of Mathematics and Computer Sciences,\\
Mohammed 1st University,\\
Oujda - Morocco,\\
abdelmalekazizi@yahoo.fr.
\end{document}